\documentclass{article}

\usepackage{amsthm,amsmath,amsfonts,amssymb}
\usepackage[a4paper,portrait]{geometry}
\usepackage{tikz}
\usepackage{authblk}
\usepackage{hyperref}

\theoremstyle{plain}

\newtheorem{theorem}{Theorem}[section]
\newtheorem{proposition}[theorem]{Proposition}
\newtheorem{corollary}[theorem]{Corollary}
\newtheorem{lemma}[theorem]{Lemma}

\newcommand{\N}{{\mathbb N}}
\newcommand{\E}{{\mathbb E}}
\renewcommand{\P}{{\mathbb P}}
\newcommand{\V}{{\mathcal M}}
\newcommand{\R}{{\mathbb R}}

\begin{document}

\author[1]{Jakub Gismatullin\footnote{jakub.gismatullin@uwr.edu.pl \quad orcid: 0000-0002-4711-3075}}    \affil[1]{Instytut Matematyczny Uniwersytetu Wroc{\l}awskiego, pl. Grunwaldzki 2/4, 50-384 Wroc{\l}aw}
\author[1,2]{Patrick Tardivel\footnote{Patrick.Tardivel@u-bourgogne.fr \quad  orcid: 0000-0002-8496-3909}} \affil[2]{Institut de Mathématiques de Bourgogne,
UMR 5584 CNRS, Université de Bourgogne,  Dijon, France}

\title{Beta distribution and associated Stirling numbers of the second kind}

\date{}

\maketitle

\begin{abstract}
This article gives a formula for associated Stirling numbers of the second kind  based on the moment of a sum of independent random variables having a beta distribution. From this formula we deduce, using probabilistic approaches, lower and upper bounds for these numbers.
\end{abstract}

\section{Introduction}
Classical Stirling numbers of the second kind $S(p,m)$ counts the number of all partitions of 
$\{1,\dots,p\}$ into $m$ nonempty subsets, for $p\in\N_{>0}$ and $ m\in\N$. More generally, the $r-$associated Stirling number $S_r(p, m)$, with $r\in \N_{>0}$, is the number of all partitions of 
$\{1,\dots,p\}$, into $m$ subsets where each subset contains at least $r$ elements \cite[p. 221]{comtet}. Obviously $S(p,m)=S_1(p,m)$.  Some sub-sequences of the multi-sequence $\{S_r(p,m):p,m,r\in\N_{>0},\ p\ge rm\}$ are in the On-Line Encyclopedia of Integer Sequences (OEIS) \cite{oeis}. Specifically, arrays  $\{S_1(p,m)\}_{\{p,m\}}$,   $\{S_2(p,m)\}_{\{p,m\}}$,
$\{S_3(p,m)\}_{\{p,m\}}$   are identified in the OEIS  by
\href{http://oeis.org/A008277}{A008277},  \href{http://oeis.org/A008299}{A008299},  \href{http://oeis.org/A059022}{A059022}. Moreover, sequences
$\{S_2(k + 6, 3)\}_{\{k\}}$, $\{S_2(k + 8, 4)\}_{\{k\}}$, representing the	number of ways of placing $k+6$ or $k+8$ labelled balls into 3 or 4 indistinguishable boxes with at least 2 balls in each box  are identified  in the OEIS  by
\href{http://oeis.org/A000478}{A000478}, \href{https://oeis.org/A058844}{A058844}.

There are well-known connections between Stirling numbers of the second kind and probability theory. 
For example, sequences $S_1(p,m)$ and $S_2(p,m)$ are asymptotically  normal when $p$ tends to $+\infty$ \cite{har,czab}. More precisely,
when $r\in \{1,2\}$,  the following convergence in distribution occurs
\[\frac{Y_p-\E(Y_p)}{\sqrt{{\rm var}(Y_p)}}\overset{d}{\underset{p\to +\infty}{\longrightarrow}}\mathcal{N}(0,1)\text{ where } \P(Y_p=m)=\frac{S_r(p,m)}{\sum_{k=1}^{p}S_r(p,k)}, \text{ for all }m\in \N_{>0}.\] 
Otherwise, according to Dobiński's formula,
 the moment of order $p$ of a Poisson distribution with parameter $\lambda\ge 0$ is $\sum_{m=1}^{p}S_1(p,m)\lambda^m$ (see, e.g. \cite[p. 211]{comtet}). 
However, to our knowledge, there is no close formula in the literature for $S_r(k,m)$ based on moments of a sum of independent and identically distributed (i.i.d) random variables. The main result in this article is Theorem \ref{thm:identity} 
providing the following new identity
\begin{equation}\label{eq:id}
S_r(p,m)=\frac{p!}{m!(r!)^m(p-rm)!}
\E\left[\left(X_1+\dots+X_m\right)^{p-rm}\right]
\end{equation}
where $X_1,\dots,X_p$ are i.i.d random variables having beta distribution with parameter $(1,r)$. 
Note that the beta distribution with parameter $(1,1)$ is a uniform distribution on $[0,1]$. Thus, when $r=1$, 
the above formula  is quite simple:
\[S_1(p,m)=\binom{p}{m}
\E\left(Z^{p-m}\right)\] 
 where $Z=\sum_{i=1}^m X_i$ has the Irwin-Hall distribution on $[0,m]$. 
Propositions \ref{prop:bounds_Jansen}, 
\ref{prop:asym} and \ref{prop:bounds_expo} give  upper and lower bounds for  $\E\left[\left(X_1+\dots+X_m\right)^{p-rm}\right]$. These bounds are sharp when one parameter $m$, $r$, or $p-rm$ tends to $+\infty$ and provides thus accurate approximations of $r-$associated Stirling numbers. 
\section{Closed formula for Stirling numbers and moments of random variables}
The density $g_r$ of a beta  distribution with parameters $(1,r)$ where $r\in \N_{>0}$ is 
\begin{equation} \label{beta}
g_r(x) = \begin{cases}
r(1-x)^{r-1} &\text{if }x\in[0,1]\\
0 &\text{otherwise}
\end{cases}.
\end{equation}
Let $X_1,\dots, X_m$ be independent random variables having the same beta $(1,r)$ distribution. The moment of order $k\in\N$ of the sum of these variables is defined as follows
\begin{equation}
\label{eq:mom0}
\V_r(k,m)=\E\left[\left(X_1+\dots+X_m\right)^k\right].
\end{equation}
Theorem \ref{thm:identity} provides a closed formula for the Stirling numbers of the second kind based on the moment $\V_r(k,m)$. 
\begin{theorem}
\label{thm:identity} Let $m,r\in \N_{>0}$ and $p\in \N$ where $p\ge rm$. The Stirling numbers of the second kind satisfy the following identity
\begin{equation}
S_r(p,m)=\frac{p!}{m!(r!)^m(p-rm)!}
\V_r(p-rm,m).
\end{equation}
\end{theorem}
From Theorem \ref{thm:identity}, one may deduce that $\E(Z^k)=S_1(m+k,m)/\binom{m+k}{m}$ where $Z$ has an Irwin-Hall's distribution on $[0,m]$ \cite{irwin1927frequency,hall1927distribution}. Note that the moment generating function of $Z$ is $\sum_{k\ge 0}\E(Z^k)t^k/k!=((\exp(t)-1)/t)^m$ therefore we recover the well-known exponential generating function of the Stirling numbers of the second kind $\sum_{p\ge m}S_1(p,m)t^p/p!=(\exp(t)-1)^m/m!$ 
(see Theorem 3.3 page 52 in \cite{mansour2015commutation}).
 The above expression of $S_r(p,m)$ is explicit up to the computation of the moment $\V_r(k,m)$. 
Whereas computing explicitly $\V_r(k,m)$ might be technical,
 lower bounds, upper bounds and approximations of $\V_r(k,m)$ are tractable as illustrated in the following section.

\section{Upper and lower bounds}

Hereafter, we will use probabilistic approaches to derive upper and lower bounds for the moment $\V_r(k,m)$. 
 \subsection{Sharp upper and lower bounds when $m$ is large}
 Let $\overline{X}_m=(X_1+\dots+X_m)/m$, the Jensen's inequality provides the following lower bound 
 \begin{equation}\label{lower}
 \V_r(k,m)=m^k \E\left({\overline{X}_m}^k\right) \ge m^k \E\left({\overline{X}_m}\right)^k =\frac{m^k}{(r+1)^k}.
 \end{equation}
 This inequality relies on the linearization of the function $q(x)=x^k$ at $x_0=\E(\overline{X}_m)=\frac{1}{r+1}$. Specifically, the following inequality holds for all $x\in [0,1]$
\begin{equation}
\label{eq:Jansen_LB}
 x^k=q(x)\ge q(x_0) + q'(x_0)(x-x_0) = \frac{1}{(r+1)^k}+\frac{k}{(r+1)^{k-1}}\left(x-\frac{1}{r+1}\right).
\end{equation}
Moreover one may choose $c \ge 0$ for which the following inequality is true for all $x\in [0,1]$ (see Lemma \ref{lem:c})
\begin{equation}
\label{eq:Jansen_UB}
 x^k\le q(x_0) + q'(x_0)(x-x_0) + c (x-x_0)^2.
\end{equation}

Proposition \ref{prop:bounds_Jansen} below is a consequence of inequalities \eqref{eq:Jansen_LB} and \eqref{eq:Jansen_UB}. 

\begin{proposition}
\label{prop:bounds_Jansen}
Let $r,m\in \N_{>0}$ and $k\in \N$. The following inequality holds
 \begin{equation}
  \frac{m^k}{(r+1)^k}\le \V_r(k,m) \le 
  \frac{m^k}{(r+1)^k} + m^{k-1}\frac{(r + 1)^k - 1 - k r}{(r + 1)^k r (r+2)}.
 \end{equation}
\end{proposition}

The leading term when $m$ is large in both the lower and upper bounds is $m^k/(r+1)^k$. Therefore, lower and upper bounds are asymptotically equivalent when $k\in \N_{\ge 0}$ and $r\in \N_{>0}$ are fixed and 
$m$ tends to $+\infty$. These bounds are accurate
when $m$ is large since $\overline{X}_m$ converges to $\E(\overline{X}_m)$ and both inequalities \eqref{eq:Jansen_LB} and  \eqref{eq:Jansen_UB}
are accurate on the neighbourhood of $\E(\overline{X}_m)$.

\subsection{Sharp upper and lower bounds when $k$ is large}

Asymptotic behaviour of moments, when $k$ is large, depends on the density of $X_1+\dots+X_m$ on the tail, \emph{i.e.} on the neighbourhood of $m$. This motivates us to introduce the following inequality proved in  Corollary \ref{cor:tail}: 
\begin{equation}
g_r^{*m}(x)\le \frac{(r!)^m}{(mr-1)!}(m-x)^{mr-1}, \text{ for all }x\in [0,m]
\end{equation} 
where $g_r$ is given by \eqref{beta} and $g_r^{*m}$ is $m$-th convolution of $g_r$. Moreover, this inequality is an equality for $x\in [m-1,m]$. We derive from this fact a lower and upper bounds for $\V_r(k,m)$ given in Proposition \ref{prop:asym} below. 

\begin{proposition}\label{prop:asym}
For any $r\in \N_{>0}$, for any $k\in \N$ for any $m\in \N_{>0}$
the following inequalities hold.
\begin{eqnarray*}
 \V_{r}(k,m) &\le&  \frac{(r!)^m}{(mr-1)!}\int_{0}^{m}x^{k}(m-x)^{mr-1}dx = \frac{k!(r!)^m m^{k+rm}}{(k+mr)!}.\\
\V_{r}(k,m) &\ge& \frac{(r!)^m}{(mr-1)!}\int_{m-1}^{m}x^{k}(m-x)^{mr-1}dx\\ 
 &\ge & \frac{k!(r!)^m m^{k+mr} }{(k+mr)!}\left(1 - \frac{(m-1)^k}{m^{k+mr}}\sum_{i=1}^{mr} \binom{k+mr}{k+i}(m-1)^{i} \right).
\end{eqnarray*}
These bounds are accurate when $k$ is large since $$\lim_{k\to+\infty} \frac{(m-1)^k}{m^{k+mr}}\sum_{i=1}^{mr} \binom{k+mr}{k+i}(m-1)^{i}=0.$$
\end{proposition}

As a consequence of Proposition \ref{prop:asym}, we observe that $S_{r}(p,m)\le m^p/m!$. Moreover, lower and upper bounds are asymptotically equivalent when $k$ tends to $+\infty$ therefore $S_r(p,m)\sim m^{p}/m!$ when $p$ is large. This approximation, well known when $r=1$ (see \cite{asym}), remains true when $r>1$. 

\subsection{Sharp upper bound when $r$ is large}
Proposition \ref{prop:bounds_expo} proves that the moment $\V_r(k,m)$
is bounded, up to an explicit expression, by the moment of a sum of independent random variables having the same standard exponential distribution.
\begin{proposition}
\label{prop:bounds_expo}
Let $k \in \N$, $m\in \N_{>0}$, $r\in \N_{>1}$ and $\mathcal{E}_1,\dots,\mathcal{E}_m$ be i.i.d random variables having standard exponential distribution with density $\exp(-x)$.
\begin{itemize}
    \item[i)] The following inequality holds
    \begin{equation}
    r^k\V_r(k,m)\le \left(\frac{r}{r-1}\right)^{2k} \E\left(\mathcal{E}_1+\dots+\mathcal{E}_m\right)^k =\left(\frac{r}{r-1}\right)^{2k} \frac{(m-1+k)!}{(m-1)!}.
    \end{equation}
    \item[ii)] The upper bound given in i) is sharp since the following limit holds 
    \begin{equation}
    \lim_{r\to +\infty} r^{k}\V_r(k,m)=\E\left(\mathcal{E}_1+\dots+\mathcal{E}_m\right)^k= \frac{(m-1+k)!}{(m-1)!}.
    \end{equation}
\end{itemize}
\end{proposition}
It seems difficult for authors to find lower bounds which are sharp when $r$ tends to $+\infty$. Finally, we recap hereafter lower and upper bounds for $r-$associated Stirling number of the second kind:
\begin{itemize}
\item Proposition \ref{prop:bounds_Jansen} provides the following lower and upper bounds
\begin{equation}
\begin{cases}
   S_r(p,m) \ge \frac{p!m^{p-rm}}{m!(r!)^m(p-rm)!(r+1)^{p-rm}}\\
  S_r(p,m) \le \frac{p!m^{p-rm}}{m!(r!)^m(p-rm)!(r+1)^{p-rm}}\left(1+\frac{(r+1)^{p-rm}-1-r(p-rm)}{mr(r+2)} \right).
\end{cases}
\end{equation}
These bound are equivalent when $p-rm$, $r$ are fixed and when $m$ tends to $+\infty$. 
\item Proposition \ref{prop:asym} provides the following lower and upper bounds
\begin{equation}
\begin{cases}
   S_r(p,m) \ge \frac{m^p}{m!} - \frac{(m-1)^{p-rm}}{m!}\sum_{i=1}^{mr} \binom{p}{p-rm+i}(m-1)^{i}\\
  S_r(p,m) \le \frac{m^p}{m!}.
\end{cases}
\end{equation}
These bound are equivalent when $m, r$ are fixed and when $p$ tends to $+\infty$. 
 \item Proposition \ref{prop:bounds_expo} provides the following  upper bound when $r\ge 2$
\begin{equation}
S_r(p,m)\le \frac{p!r^{2(p-rm)}(m-1+p-rm)!}{m!(r!)^m(p-rm)!(r-1)^{2(p-rm)}(m-1)!}.
\end{equation}
This upper bound is equivalent to $S_r(p,m)$ when $p-rm, m$ are fixed and when $r$ 
tends to $+\infty$. 
\end{itemize}

\section{Numerical experiments}
\subsection{Upper and lower bounds of Stirling numbers of the second kind}

According to Propositions \ref{prop:bounds_Jansen} and \ref{prop:asym}, for all $m\in \N$ and all $p\in \N_{>0}$, Stirling numbers of the second kind do satisfy the following inequalities: 
\begin{eqnarray*}
S_1(p,m)&\le& \underbrace{\min\left\{\frac{m^p}{m!}, \binom{p}{m}\left(\frac{m}{2}\right)^{p-m}\left(1+\frac{2^{p-m}+m-p-1}{3m}\right)\right\}}_{U(p,m)}\\
S_1(p,m)&\ge& \underbrace{\max\left\{\frac{m^p}{m!} - \frac{(m-1)^{p-m}}{m!}\sum_{i=1}^{m} \binom{p}{m-i}(m-1)^{i}, \binom{p}{m}\left(\frac{m}{2}\right)^{p-m}\right\}}_{L(p,m)}.
\end{eqnarray*}
First of all we are going to compare these bounds $U(m,p)$ and  $L(p,m)$ to bounds 
given in Rennie and Dobson \cite{rennie1969stirling} reported below:
\begin{equation}
\underbrace{\frac{1}{2}(m^2+m+2)m^{p-m-1}-1}_{L_{\rm rd}(p,m)}\le S_1(p,m) \le \underbrace{\frac{1}{2}\binom{p}{m}m^{p-m}}_{U_{\rm rd}(p,m)}.
\end{equation}
Numerical comparison between $U(p,m)$ and $ U_{\rm rd}(p,m)$ is not relevant since  
\[
\frac{1}{2}\binom{p}{m}m^{p-m} \ge \binom{p}{m}\left(\frac{m}{2}\right)^{p-m}\left(1+\frac{2^{p-m}+m-p-1}{3m}\right)
\] 
holds for $m<p$. Contrarily to the upper bound, the lower bound $L(p,m)$ is not uniformly larger than the one given by Rennie and Dobson; for instance, $31=L_{\rm rd}(6,2)> L(6,2)=28.5$. Numerical experiments in Figure \ref{fig:Figure_1} illustrate that for most integers $p,m$, $L(p,m)$ is a better approximation of $S_1(p,m)$ than $L_{\rm rd}(p,m)$. 

\begin{figure}
\centering
\includegraphics[scale=0.6]{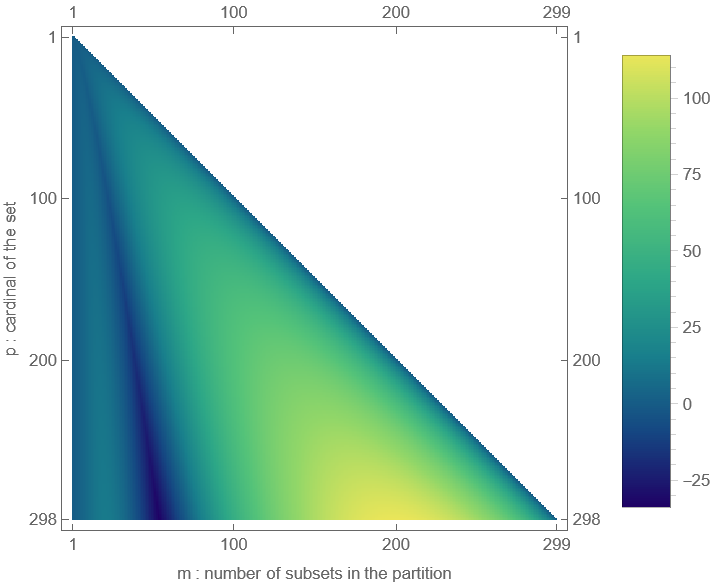}
\caption{This figure report  $\ln(L(p,m))-\ln(L_{\rm rd}(p,m))$
    as a function of $m$ (on the $x-$axis) and $p$ (on the $y-$axis). One may observe that for most integers the lower bound $L(p,m)$ is a better approximation of $S_1(p,m)$ than $L_{\rm rd}(p,m)$ (as  $\ln(L(p,m))-\ln(L_{\rm rd}(p,m))>0$).}
\label{fig:Figure_1}
\end{figure}

Figure \ref{fig:Figure_2} provides a comparison between $L(p,m), U(p,m)$ and $S_1(p,m)$.

\begin{figure}
\centering
\begin{tabular}{c c}
      \includegraphics[scale=0.5]{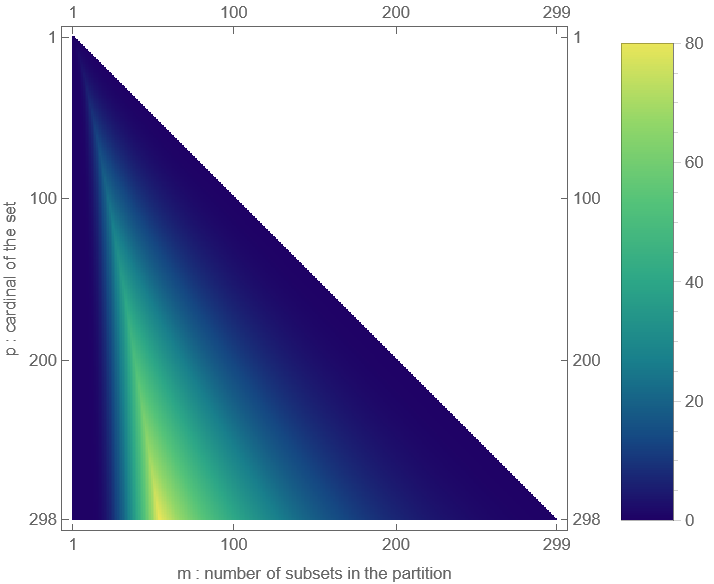} &
       \includegraphics[scale=0.5]{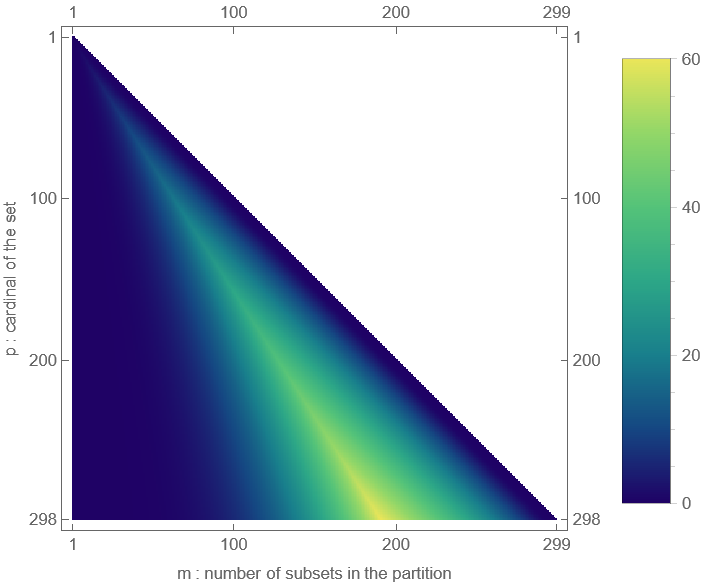} 
    \end{tabular}
\caption{ This figure report  $\ln(S_1(p,m))-\ln(L(p,m))$ (on the left) and $\ln(U(p,m))-\ln(S_1(p,m))$ (on the right) as a function of $m$  and $p$.  These numerical experiments comply with Propositions \ref{prop:bounds_Jansen} and \ref{prop:asym} since both lower and upper bounds accurately approximate $S_1(p,m)$ when $p$ is large and $m$ is small or when $m$ is large and $p-m$ is small.}
\label{fig:Figure_2}
    
\end{figure}

\subsection{Upper bounds of Bell numbers}
The Bell number $B(p)$, where $p\in \N_{>0}$, represents the number of partition of $\{1,\dots,p\}$. Since a Bell number is a sum of Stirling numbers of the second kind 
$B(p)= \sum_{m=1}^{p}S_1(p,m)$ then, the following inequality occurs:
\begin{equation}
B(p)\le \underbrace{\sum_{m=0}^{p}
\min\left\{\frac{m^p}{m!}, \binom{p}{m}\left(\frac{m}{2}\right)^{p-m}\left(1+\frac{2^{p-m}+m-p-1}{3m}\right)\right\}}_{=U(p)}
\end{equation}
for all $p\in \N_{>0}$. 
In Figure \ref{fig:Figure_3} we compare $U(p)$ with the upper bound: 
$B(p)\le U_{\rm bt}(p) = \left( \frac{0.792p}{\ln(p+1)} \right)^p$ given in Berend and Tassa \cite{berend2010improved}.

\begin{figure}
\centering
\includegraphics[scale=0.5]{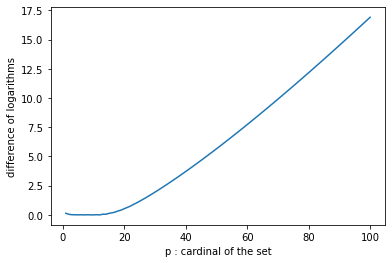}
\caption{This figure report  $\ln(U_{bt}(p))-\ln(U(p))$
    as a function of $p$. One may observe that when $p\ge 13$, $U(p)$ is more accurate upper bound for $S_1(p,m)$ than $U_{\rm bt}(p)$ (as $\ln(U_{bt}(p))-\ln(U(p))>0$ for $p\ge 13$). }
\label{fig:Figure_3}
\end{figure}

Note that $U(p)\le \sum_{m=0}^{+\infty}m^p/m!=eB(p)$ (the last equality is due to the Dobi\'nski formula). In Figure \ref{fig:Figure_4}
we show that $U(p)/B(p)$ is very close to $e$ when $p$ is large.
\begin{figure}
\centering
\includegraphics[scale=0.5]{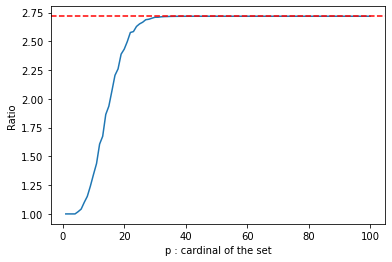}
\caption{This figure report  $U(p)/B(p)$
    as a function of $p$. One may observe that $U(p)/B(p)$ is approximately equal to $e$ when $p$ is large. }
\label{fig:Figure_4}
\end{figure}

\section{Proofs}

\subsection{Proof of Theorem \ref{thm:identity}}

The identity given in Lemma \ref{lem:identity} below, combined with the 
multinomial formula, allows us to complete the proof of Theorem \ref{thm:identity}.

\begin{lemma}
\label{lem:identity}
Let $m,r\in \N_{>0}$ and $p\in \N$ where $p\ge rm$. The $r$-associated Stirling numbers of the second kind satisfy the following equality
\[S_r(p,m)=\frac{p!}{m!}\sum_{i_1+\dots+i_m=p-rm}\frac{1}{(r+i_1)!\times\ldots\times(r+i_m)!}\]
where the sum is computed over all the integers $i_1,\dots,i_m\in \{0,\dots,p-rm\}$ satisfying $i_1+\dots+i_m=p-rm$.
\end{lemma}
\begin{proof}
Given $i_1,\dots,i_m\in \N$  such that 
$i_1+\dots+i_m=p-rm$, let us count the number of ordered partitions 
of $\{1,\dots,p\}$
in $m$ parts where the first part has $r+i_1$ elements,
the second part has $r+i_2$ elements and so on.

There are $\binom{p}{r+i_1}$ possibilities for the first part.
 There are $\binom{p-r-i_1}{r+i_2}$ possibilities for the second part and so on. Therefore the number of ordered partitions where the first part has $r+i_1$ elements,
 the second part has $r+i_2$ elements and so on is 
\[\frac{p!}{(r+i_1)!\times\ldots\times (r+i_m)!}.\]
Consequently, the number of ordered partitions of $\{1,\dots,p\}$ in $m$ parts having at least $r$ elements is 
\[
\sum_{i_1+\dots+i_m=p-rm}\frac{p!}{(r+i_1)!\times\ldots\times(r+i_m)!}.
\]
Finally, when the order is not taken into account, by dividing by $m!$, 
one may deduce that
\[
S_r(p,m)=\frac{p!}{m!}\sum_{i_1+\dots+i_m=p-rm}\frac{1}{(r+i_1)!\times\ldots\times(r+i_m)!}.  
\] 
\end{proof}

\begin{proof}{Theorem {\ref{thm:identity}}}Let us recall the multinomial formula. Given $x_1,\dots,x_m\in \R$ and $k\in \N$, we have 
\[
(x_1+\dots+x_m)^{k}=\sum_{i_1+\dots+i_m=k}\frac{k!}{i_1! \times \dots \times i_m!}x_1^{i_1}\times\ldots\times x_m^{i_m}.
\]
Let $k=p-r m$. Since $\E(X_1^s)=\frac{s!r!}{(s+r)!}$, the multinomial formula and Lemma \ref{lem:identity} give 
\begin{eqnarray*}\E\left[ \left(X_1+\dots+X_m\right)^k\right] &=& \sum_{i_1+\dots+i_m=k}\frac{k!}{i_1!\times \dots \times i_m!} \E(X_1^{i_1})\times\ldots\times \E(X_m^{i_m})\\
&=& (r!)^mk!\sum_{i_1+\dots+i_m=k}\frac{1}{(r+i_1)!\times \dots \times (r+i_m)!}\\
&=&\frac{m!(r!)^m k!}{(k+rm)!}S_r(k+rm,m)\\
&=&\frac{m!(r!)^m (p-rm)!}{p!}S_r(p,m) 
\end{eqnarray*}
which finishes the proof. 
\end{proof}

\subsection{Proof of Proposition \ref{prop:bounds_Jansen}}

Proposition \ref{prop:bounds_Jansen} is a consequence of the following lemma.

\begin{lemma}\label{lem:c}
Let $k\ge 2$, $a\in (0,1)$ and $f\colon x\in [0,1]\to a^k+ka^{k-1}(x-a)+c(x-a)^2$
where $c\ge 0$ is such that $f(1)=1$ (namely $c = \frac{a^{k-1}(ak - a - k)+1}{(1-a)^2}$). Then $f(x)\ge x^k$ for all $x\in[0,1]$.
\end{lemma}

\begin{figure}
\centering
\begin{tikzpicture}[scale=5]
  \draw[->] (-0.05,0) -- (1.05,0) node[right]{$x$};
  \draw[->] (0,-0.05) -- (0,1.05) node[above]{$y$};
   \draw[samples=200,domain=0:1] plot(\x,{(\x)^3)});
   \draw[samples=200,domain=0:1,color=blue] plot(\x,{0.027+ 0.27*(\x-0.3)+ 1.6*(\x-0.3)^2)});
\draw (0.5,0.12) node[right]{$x\mapsto x^k$};
\draw (0.8,0.58) node[left,color=blue]{$x\mapsto f(x)$};
\draw (0.3,0.015) -- (0.3,-0.015);
\draw (0.3,-0.015) node[below]{$a$};

\draw (1,0.015) -- (1,-0.015);
\draw (1,-0.015) node[below]{$1$};

\draw (0.015,1) -- (-0.015,1);
\draw (-0.015,1) node[left]{$1$}; 
\end{tikzpicture}
\caption{Illustration of the inequality given in Lemma \ref{lem:c}.}
\label{fig:illustration_ineq}
\end{figure}
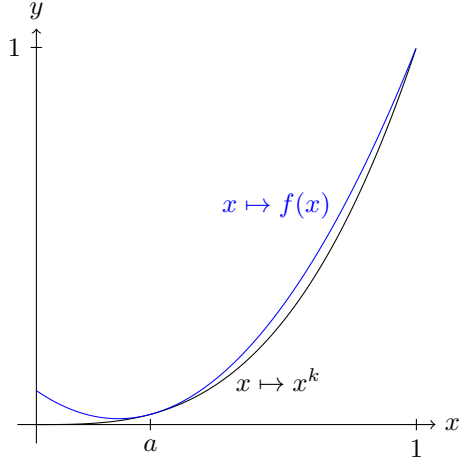
\begin{proof}
First note that for all $x\in[0,1]$ the condition $f(x)\ge x^k$ is equivalent to 
\[
ka^{k-1}(x-a)+c(x-a)^2 \ge x^k-a^k = (x-a)(x^{k-1}+ax^{k-2}+\dots+a^{k-1}).
\]
Note that this inequality holds if and only if 
\begin{eqnarray*}
ka^{k-1}+c(x-a)&\ge& x^{k-1}+ax^{k-2}+\dots+a^{k-1}\text{ for all }x\in [a,1] \text{ and}\\
ka^{k-1}+c(x-a)&\le& x^{k-1}+ax^{k-2}+\dots+a^{k-1}\text{ for all } x\in [0,a].
\end{eqnarray*}
Let $d(x)=ka^{k-1}+c(x-a)$ and $p(x)=x^{k-1}+ax^{k-2}+\dots+a^{k-1}$. Thus $p(a)=d(a)$ and $p(1)=d(1)$, by construction of $c$. Because $p$ is convex and $d$ is affine, one may deduce that $d(x)\le p(x)$ once $x\in [0,a]$ and $d(x)\ge p(x)$ once $x\in [a,1]$, which completes the proof. 
\end{proof}

\begin{proof}{Proposition {\ref{prop:bounds_Jansen}}}
By Lemma \ref{lem:c}, for $a=\E(\overline{X}_m)=\frac{1}{r+1}$, $r>0$, for all $x\in [0,1]$ we get that
\begin{eqnarray*}
x^k &\le& \E(\overline{X}_m)^k+k\E(\overline{X}_m)^{k-1}\left(x-\E(\overline{X}_m)\right)+ c\left(x-\E(\overline{X}_m)\right)^2 \\
&\le& \frac{1}{(r+1)^k}+\frac{k}{(r+1)^{k-1}}\left(x-\frac{1}{r+1}\right) + c\left(x-\frac{1}{r+1}\right)^2 
\end{eqnarray*} 
where $c=(1+\frac{1}{r})^2\left(1-\frac{k r+1}{(r+1)^k} \right)$. This inequality implies that 
\begin{eqnarray*}
\V_r(k,m)=m^k\E(\overline{X}_m^k) &\le& m^k\left(\E(\overline{X}_m)^k+c {\rm var}(\overline{X}_m)\right) \\
&\le & \frac{m^k}{(1+r)^k}+c\frac{rm^{k-1}}{(1+r)^2(2+r)}\\
&\le & \frac{m^k}{(r+1)^k} + \frac{(r + 1)^k - 1 - k r}{(r + 1)^k r (r+2)} m^{k-1}.
\end{eqnarray*}
\end{proof}

\subsection{Proof of Proposition \ref{prop:asym}}

We use in the proof a well-known beta integral (see, for example, \cite{hbeta}): let $a, b\in \N$ and $x\in\R_{>0}$ then we have 
\begin{equation} \label{intbeta}
\int_{0}^{x}(x-t)^a t^b dt=\frac{a!b!}{(a+b+1)!}x^{a+b+1}.
\end{equation}
To compute explicitly the density of $X_1+\dots+X_m$ on the tail $[m-1,m]$, we use the following technical lemma. 
Let $h_r$ be the following density
\begin{equation} \label{hbeta}
h_r(x) = \begin{cases}
r x^{r-1} &\text{if }x\in[0,1]\\
0 &\text{otherwise}
\end{cases}.
\end{equation}
Actually, $h_r$ is the density of $1-X$ where the density $X$ is $g_r$ \eqref{beta}. Convolution computations are slightly easier to handle with $h_r$ than $g_r$.

\begin{lemma}\label{lem:h}
Let $m\in \N_{>0}$ then the following equality and inequality hold
\begin{eqnarray}\label{eq1}
h_r^{*m}(x) &=\frac{(r!)^m}{(mr-1)!}x^{mr-1}, \text{ for all }x\in [0,1].\\ \label{eq2}
h_r^{*m}(x) &\le \frac{(r!)^m}{(mr-1)!}x^{mr-1}, \text{ for all }x\in \R_{\ge 0}.
\end{eqnarray}
\end{lemma}
\begin{proof}  

Let us prove \eqref{eq1} by induction.
 When $m=1$, one may notice that whatever $x\in [0,1]$ we have $h_r^{*m}(x)=h_r(x)$.
Let $m\in \N_{>0}$ such that
\[h_r^{*m}(x)=\frac{(r!)^m}{(mr-1)!}x^{mr-1} \quad \forall x\in [0,1].\]
Therefore for $x\in [0,1]$ we have that
\begin{eqnarray*}
h_r^{*m+1}(x)&=&\int_{\R}h^{*m}_r(x-t)h_r(t)dt=\int_{0}^{x}h^{*m}_r(x-t)rt^{r-1}dt\\
&=&\frac{r(r!)^m}{(mr-1)!}\int_{0}^{x}(x-t)^{mr-1}t^{r-1}dt\\
&=& \frac{r(r!)^m}{(mr-1)!}\frac{(mr-1)!(r-1)!}{((m+1)r-1)!}x^{(m+1)r-1}\\
&=&\frac{(r!)^{m+1}}{((m+1)r-1)!}x^{(m+1)r-1}.
\end{eqnarray*}
The proof of \eqref{eq2} by induction is quite similar than the proof of \eqref{eq1}. When $m=1$, the result is straightforward. 
Let $m\in \N_{>0}$ such that
\[h_r^{*m}(x)\le \frac{(r!)^m}{(mr-1)!}x^{mr-1} \quad \forall x\in  \R_{\ge 0}.\]
Therefore for $x\in \R_{\ge 0}$ we have 
\begin{eqnarray*}
h_r^{*m+1}(x)&=&\int_{\R}h^{*m}_r(x-t)h_r(t)dt=\int_{0}^{x}h^{*m}_r(x-t)h_r(t)dt\\
&\le &\frac{r(r!)^m}{(mr-1)!}\int_{0}^{x}(x-t)^{mr-1}t^{r-1}dt\\
&\le & \frac{r(r!)^m}{(mr-1)!}\frac{(mr-1)!(r-1)!}{((m+1)r-1)!}x^{(m+1)r-1}\\
&\le &\frac{(r!)^{m+1}}{((m+1)r-1)!}x^{(m+1)r-1}. 
\end{eqnarray*} 
\end{proof}

Note that the only difference between the proof of \eqref{eq1} and \eqref{eq2} is the majorization of $h_r(t)$ by $rt^{r-1}$. 
We are now ready to prove the explicit formula for the $m$-th convolution $g_r^{*m}$ on $[m-1,m]$ and an upper bound on $[0,m]$.

\begin{corollary}
\label{cor:tail}
For all $x\in [0,m]$ we have that
\[g_r^{*m}(x)\le \frac{(r!)^m}{(mr-1)!}(m-x)^{mr-1}.\]
Moreover if $x\in [m-1,m]$ then 
\[g_r^{*m}(x)=\frac{(r!)^m}{(mr-1)!}(m-x)^{mr-1}.\]
\end{corollary}
\begin{proof}
By Lemma \ref{lem:h}, it  suffices to prove that $g_r^{*m}(x)=h_r^{*m}(m-x)$ for all $x\in[0,m]$.
Let $x\in [0,m]$ and set $z=m-x$, because the density of $(1-X_1)+\dots +(1-X_m)$ is
 $h_r^{*m}$  then the  following equalities occur
\begin{eqnarray*}
\frac{\partial}{\partial z}\P((1-X_1)+\dots +(1-X_m)\le z) &=& h_r^{*m}(z)\\
\frac{\partial}{\partial z}\P(X_1+\dots+X_m\ge m-z) &=& h_r^{*m}(z)\\
g_r^{*m}(m-z) &=& h_r^{*m}(z)\\
g_r^{*m}(x) &=& h_r^{*m}(m-x).
\end{eqnarray*}
\end{proof}

One may notice that if $r=1$ and $x\in [m-1,m]$, then $g_r^{*m}(x)=(m-x)^{m-1}/(m-1)!$ which is the density in the tail of the Irvin-Hall distribution (see \cite{hall1927distribution,irwin1927frequency}). 

Upper and lower bounds given in Proposition \ref{prop:asym} are straightforward consequences of  Corollary \ref{cor:tail}. Indeed, the upper bound is just the beta integral of $$\int_{0}^{m}\frac{(r!)^m}{(mr-1)!}(m-x)^{mr-1}x^kdx.$$ 
\begin{proof}{Proposition {\ref{prop:asym}}} The lower bound is derived from the following computations 
\begin{eqnarray*}
&&\frac{(r!)^m}{(mr-1)!}\int_{0}^{m-1}x^{k}(m-x)^{mr-1}dx \\
&=& \frac{(r!)^m}{(mr-1)!}\int_{0}^{m-1}x^{k}(1+(m-1)-x)^{mr-1}dx \\ &=&\frac{(r!)^m}{(mr-1)!}\int_{0}^{m-1}x^{k}\sum_{i=0}^{mr-1}\binom{mr-1}{i} (m-1-x)^{i}dx \\
&=& \frac{(r!)^m}{(mr-1)!}\sum_{i=0}^{mr-1} \binom{mr-1}{i} \frac{k!i!(m-1)^{k+i+1}}{(k+i+1)!}  \\
&=&  \frac{(r!)^mk!}{(k+mr)!} \sum_{i=1}^{mr} \binom{k+mr}{k+i}(m-1)^{k+i}.
\end{eqnarray*}

Because $ \V_{r}(k,m) \ge \frac{(r!)^m}{(mr-1)!}\int_{m-1}^{m}x^{k}(m-x)^{mr-1}dx$ one may deduce the following inequality
\[\V_{r}(k,m) 
 \ge   \frac{k!(r!)^m m^{k+mr} }{(k+mr)!}\left(1 - \frac{(m-1)^k}{m^{k+mr}}\sum_{i=1}^{mr} \binom{k+mr}{k+i}(m-1)^{i} \right).
\]

Let us assume that $k>mr$, then $\binom{k+mr}{k+i}\le \binom{k+mr}{k}$, for $1\le i\le mr$, hence
\begin{eqnarray*}
&&\frac{(m-1)^k}{m^{k+mr}}\sum_{i=1}^{mr} \binom{k+mr}{k+i}(m-1)^{i} \\
&\le& \frac{1}{m^{mr}}\left(1-\frac{1}{m}\right)^k \sum_{i=1}^{mr}\binom{k+mr}{k}  (m-1)^{i} \\
&\le& \left(1-\frac{1}{m}\right)^k \binom{k+mr}{k} \frac{((m-1)^{m r+1}-(m-1)}{m-2}  \\ 
&\le& \left(1-\frac{1}{m}\right)^k (k+mr)^{mr} \frac{(m-1)^{2mr}}{(mr)!} \underset{k \to +\infty}{\longrightarrow} 0.
\end{eqnarray*}
\end{proof}

\subsection{Proof of Proposition \ref{prop:bounds_expo}}

\begin{proof}{Proposition {\ref{prop:bounds_expo}}}
i) Note that the density of the random variable $Y_i=(r-1)X_i$  on $[0,r-1]$ is
\[ f(x) =\frac{r}{r-1}\left(1-\frac{x}{r-1}\right)^{r-1}.\]
Due to the following inequality for all $x\in [0,r-1]$ 
\[
\frac{r}{r-1}\left(1-\frac{x}{r-1}\right)^{r-1} \le \frac{r}{r-1}\exp(-x)
\]
one may deduce that 
\begin{eqnarray*} (r-1)^k\E\left(\sum_{i=1}^m X_i  \right)^k = \E\left(\sum_{i=1}^m Y_i  \right)^k &\le& \left(\frac{r}{r-1}\right)^k \E\left(\sum_{i=1}^m \mathcal{E}_i\right)^k\\
&\le& \left(\frac{r}{r-1}\right)^k \frac{(m-1+k)!}{(m-1)!}
\end{eqnarray*}since $\sum_{i=1}^m \mathcal{E}_i$ has an Erlang distribution  (the density of $\sum_{i=1}^m \mathcal{E}_i$ is  \\
$h(x)=\frac{x^{m-1}\exp(-x)}{(m-1)!}$) whose  moment of order $k$ is $\frac{(m-1+k)!}{(m-1)!}$. This inequality completes the proof of i).

ii) The density of the random variable $(r-1)X_1$ converges pointwise to the density  of $\mathcal{E}_1$ namely $\lim_{r\to +\infty}f(x)=\exp(-x)$, for all $x\in [0,+\infty)$. Therefore, the following limit holds
\[ 
  \lim_{r\to +\infty}x^kf^{*m}(x)=
  \frac{x^{m+k-1}\exp(-x)}{(m-1)!}.
\]
Finally, the Dominated Convergence Theorem gives that
\begin{eqnarray*}
\lim_{r\to +\infty}
(r-1)^k\E\left(\sum_{i=1}^m X_i\right)^k &=&
\lim_{r\to +\infty} \int_{0}^{+\infty} \frac{x^{m+k-1}\exp(-x)}{(m-1)!}dx\\
&=&  \E\left(\sum_{i=1}^m \mathcal{E}_i\right)^k=\frac{(m-1+k)!}{(m-1)!}.
\end{eqnarray*}
\end{proof}

\section*{Funding}
The first author is supported by The National Science Centre, Poland NCN grants no.  2014/13/D/ST1/03491 and 2017/27/B/ST1/01467.
The Institut de Mathématiques de Bourgogne (IMB) receives support from the EIPHI Graduate School (contract ANR-17-EURE-0002). The second author 
receives support from the region Bourgogne-Franche-Comté (EPADM project).

\end{document}